\documentclass[a4paper,12pt, english]{article}
\usepackage[latin1]{inputenc}
\usepackage{babel}
\usepackage{amsmath,amsfonts,amssymb,amsthm}
\usepackage{bbm}
\usepackage{hyperref}
\usepackage{verbatim}
\usepackage{tikz}

\theoremstyle{plain}
\newtheorem{teo}{Theorem}[section]
\newtheorem{cor}[teo]{Corollary}
\newtheorem{prop}[teo]{Proposition}
\newtheorem{lem}[teo]{Lemma}
\newtheorem{claim}[teo]{Claim}
\newtheorem{rem}[teo]{Remark}

\newtheorem{Q}{Question}

\theoremstyle{defin}
\newtheorem{defin}[teo]{Definition}

\newcommand{\system}[1]{\mbox{\fontfamily{cmss}\fontshape{n}\fontseries{m}\selectfont#1}}

\newcommand{\ZFC}{\system{ZFC}}
\newcommand{\AC}{\system{AC}}
\newcommand{\GCH}{\system{GCH}}

\newcommand{\DC}{\system{DC}}

\DeclareMathOperator{\cof}{cof}
\DeclareMathOperator{\crt}{crt}

\DeclareMathOperator{\Ord}{Ord}
\DeclareMathOperator{\Ult}{Ult}
\DeclareMathOperator{\ot}{ot}
\DeclareMathOperator{\dom}{dom}
\DeclareMathOperator{\ran}{ran}

\DeclareMathOperator{\lh}{lh}

\DeclareMathOperator{\mc}{mc}
\DeclareMathOperator{\supp}{supp}
\DeclareMathOperator{\Suc}{Suc}
\DeclareMathOperator{\bas}{bas}
\DeclareMathOperator{\Lev}{Lev}
\DeclareMathOperator{\Sacks}{Sacks}

\title{A general tool for consistency results related to I1}
\author{Vincenzo Dimonte\footnote{Technische Universit\"{a}t Wien, Wiedner Hauptstra\ss e 8--10, 1040 Wien, Austria \emph{E-mail address:} \texttt{vincenzo.dimonte@gmail.com},\emph{URL:}  \texttt{http://www.logic.univie.ac.at/\~{}dimonte/}}, Liuzhen Wu\footnote{Institute of Mathematics, Chinese Academy of Sciences, East Zhong Guan Cun Road No 55, Beijing 100190 China \emph{E-mail address:} \texttt{lzwu@math.ac.cn}}}

\begin{document}

\maketitle

\begin{abstract}

In this paper we provide a general tool to prove the consistency of $I1(\lambda)$ with various combinatorial properties at $\lambda$ typical at settings with $2^\lambda>\lambda^+$, that does not need a profound knowledge of the forcing notions involved. Examples of such properties are the first failure of GCH, a very good scale and the negation of the approachability property, or the tree property at $\lambda^+$ and $\lambda^{++}$.



\end{abstract}

 \section{Introduction}
 
 While Cantor gave us the means to conceive infinite cardinals, it is clear that to use them in a fruitful way a thorough study of their structure is needed, and this is the aim of the field usually called ``infinite combinatorics''. The key turning point for this study, as old as Set Theory, was the introduction of forcing \cite{Cohen}: it was clear then that much of the structural properties are independent from ZFC, therefore shifting the focus of the study from what is the structure of infinite cardinals, to what it could be. In the few years after Cohen's seminal results, the analysis of regular cardinals was pretty much complete, with the definition of many forcings that can change effortlessy the combinatorial properties not only of a single cardinal, but, using a method introduced by Easton \cite{Easton}, to all regular cardinals at once.

Changing combinatorial properties of singular cardinals, especially those of cofinality $\omega$, proved to be much harder than in the regular cardinal case. As a result, the research on such properties is rich and varied, it provided and still provides many challenges. Some situations are even impossible: Silver \cite{Silver} proved that SCH cannot fail first at a singular cardinal of uncountable cofinality, and Solovay \cite{Solovay} proved that above a strongly compact cardinal SCH must hold. The typical way to make SCH fail at a singular cardinal (i.e. blowing up the cardinality of its powerset) is to start with $\kappa$ measurable, blowing up its power and then adding an $\omega$-sequence cofinal to $\kappa$ with Prikry forcing, to make it of cofinality $\omega$. But $2^\kappa>\kappa^+$ permits a multitude of properties to hold, and it is an ongoing research to find more and more sophisticated variations of the Prikry forcing that permit different combinations of specific combinatorics on a singular 
cardinal.
 
 While this research heavily involves large cardinals, their role has almost always been giving consistency strength to a certain scenery, but they rarely appear directly with the desired combinatorial property, for the simple reason that the great majority of large cardinals are regular cardinals, therefore unrelated to the problem. Moreover some large cardinal simply do not accept a lot of variety on the structure of singular cardinals, as noted above. Going up the hierarchy, however, one can find an exception. The strongest large cardinal axioms, called rank-into-rank axioms, do involve a singular cardinal of countable cofinality, and it is very natural to question the position of them in this field: as they imply all known large cardinal hypothesis, knowing their structure is crucial in the large cardinal field, as it trickle down to all the hierarchy.  
 
 Woodin in \cite{Woodin} introduced what he called ``Generic Absoluteness for I0'', and this proved to be key for this study: starting with I0, adding a Prikry sequence to its critical point $\kappa$, an action that we noted typical for proving consistency results for singular cardinals, adds in fact I1. In \cite{DimFried} this was exploited to prove that it was possible to have $j:V_{\lambda+1}\prec V_{\lambda+1}$ and $2^\lambda>\lambda^+$ at the same time, in the same way it was proved just with a measurable.
 
 If the proof of $2^\lambda>\lambda^+$ uses Prikry forcing, could we use the sophisticated variations of the Prikry forcing that appear in literature to prove the consistency of different combinations of specific combinatorics on a singular cardinal with I1? In this paper, we extend such theorem, describing a general procedure to be applied to many of the refined Prikry forcings, therefore automatically transferring the already known results about the combinatorics of singular cardinal of countable cofinality to cardinals that moreover satisfy very large cardinal properties, providing therefore a number of new results and a tool that any researcher can use without going into the original details of the forcing notions involved.
 
 In Section 2 all the preliminary facts are collected. In Section 3, the general procedure is described: the notion of $\kappa$-geometric forcing is introduced, and this is the key notion that will permit the proof to work; the procedure is tested with Prikry forcing and tree Prikry forcing. In Section 4, the procedure is applied to the extender-based Prikry forcing, to provide $I1(\lambda)$ and the first failure of GCH at $\lambda$ (this answers a question in \cite{DimFried}). In Section 5, the procedure is applied to two different flavors of diagonal supercompact Prikry forcing, to achieve results on pcf combinatorics and the Tree Property. In Section 6, we see another application to the Tree Property. In Section 7 we note some possible directions for future research on the subject. 
 
 The first author would like to to thank the FWF (Austrian Science Fund) for its generous support through project M 1514-N25, and the kind hospitality of the Kurt G\"{o}del Research Center, Beijing Normal University and the Chinese Academy of Sciences. The second author would like to acknowledge the support through the funding projects NSFC 11321101 and 11401567.

 \section{Preliminaries}
 
 To avoid confusion or misunderstandings, all notation and standard basic results are collected here.

 Elementary embeddings have a key role in the definitions of all large cardinals from measurable to above.
 
 If $M$ and $N$ are sets or classes, $j:M\prec N$ denotes that $j$ is an elementary embedding from $M$ to $N$. We write the case in which the elementary embedding is the identity, i.e., $M$ is an elementary submodel of $N$, simply as $M\prec N$, while when $j$ is indicated we always suppose that it is not the identity.

 If $j:M\prec N$ and either $M\vDash\AC$ or $N\subseteq M$ then it moves at least one ordinal. The \emph{critical point}, $\crt(j)$, is the least ordinal moved by $j$.
 
 If $j:M\prec N$ and $N\subseteq M$, we define $j^n$ as the composition of $n$ copies of $j$, i.e., $j^1=j$ and $j^{n+1}=j\circ j^n$.

 Let $j$ be an elementary embedding and $\kappa=\crt(j)$. Define $\kappa_0=\kappa$ and $\kappa_{n+1}=j(\kappa_n)$. Then $\langle \kappa_n:n\in\omega\rangle$ is the \emph{critical sequence} of $j$.

 Kunen \cite{Kunen} proved under \AC{} that if $M=N=V_\eta$ for some ordinal $\eta\leq\Ord$, and $\lambda$ is the supremum of the critical sequence, then $\eta$ cannot be bigger than $\lambda+1$ (and of course cannot be smaller than $\lambda$).

 This at the time was considered a stop to the large cardinal study, as a $j:V\prec V$ would have been the largest possible cardinal, but Kunen's result leaves room for a new breed of large cardinal hypotheses, sometimes referred to in the literature as rank-into-rank hypotheses:
 \begin{description}
  \item[I3] iff there exists $\lambda$ s.t. $\exists j:V_\lambda\prec V_\lambda$;
  \item[I2] iff there exists $\lambda$ s.t. $\exists j:V\prec M$, with $V_\lambda\subseteq M$ and $\lambda$ is the supremum of the critical sequence;
  \item[I1] iff there exists $\lambda$ s.t. $\exists j:V_{\lambda+1}\prec V_{\lambda+1}$. 
 \end{description}
 
 The consistency order of the above hypotheses is reversed with respect to their numbering: I1 is strictly stronger than I2, which in turn is strictly stronger than I3 (see \cite{Laver}). All of these hypotheses are strictly stronger than all of the large cardinal hypotheses outside the rank-into-rank umbrella (see \cite{Kanamori}, 24.9 for $n$-huge cardinals, or \cite{Corazza2} for the Wholeness Axiom). I3 enjoyed a particularly rich literature, as it has an interesting algebraic content \cite{Dehornoy}.

 Note that if $j$ witnesses a rank-into-rank hypothesis, then $\lambda$ is uniquely determined by $j$, so in the following $\lambda$ always denotes the first nontrivial fixed point of the embedding $j$ under consideration. We write $I1(\lambda)$ for $\exists j:V_{\lambda+1}\prec V_{\lambda+1}$.
 
 Suppose that $j:V_\lambda\prec V_\lambda$ with critical sequence $\langle\kappa_n:n\in\omega\rangle$ and let $A\subseteq V_\lambda$. Then we can define $j^+(A)=\bigcup_{n\in\omega}j(A\cap V_{\kappa_n})$. Such $j^+:V_{\lambda+1}\to V_{\lambda+1}$ is a $\Sigma_0$ elementary embedding (see for example Lemma 1 in \cite{Laver2}). The key point is that $j^+$ is, in $V_{\lambda+1}$, a definable class (definable with $j\upharpoonright V_\lambda$ as a parameter). Then one can ask whether $j^+$ is stronger, i.e., a $\Sigma_n$ elementary embedding, with $n>0$. Laver proved in \cite{Laver} that this is strictly stronger than I3($\lambda$), and yet still expressible in $V_{\lambda+1}$. Now, $j^+$ is a full elementary embedding (i.e., it witnesses I1($\lambda$) if and only if it is $\Sigma_n$ for every $n$. Therefore, suppose that $j:V_{\lambda+1}\prec V_{\lambda+1}$. Then $j=(j\upharpoonright V_\lambda)^+$ and $V_{\lambda+1}$ ``knows'' that $j$, as a class defined with parameter $j\upharpoonright V_\lambda$, is an elementary embedding.
 
 \begin{rem}
 \label{I1}
 Suppose that $j:V_{\lambda+1}\to V_{\lambda+1}$. Then there are $\varphi_n(x)$ formulas such that for any $n\in\omega$, $V_{\lambda+1}\vDash\varphi_n(j\upharpoonright V_\lambda)$ iff $j$ is an elementary embedding.  
 \end{rem}

 This implies that if I1$(\lambda)$, then $L_1(V_{\lambda+1})\vDash I1(\lambda)$. Note that this is peculiar to I1, as $V_{\lambda+1}$ is not a model for \ZFC. $V_\lambda$ cannot satisfy I3($\lambda$), as $V_\lambda$ is a model for \ZFC{} and that goes against Kunen's Theorem.
 
 In the early 1980's Woodin proposed an axiom even stronger than all the previous ones:

\begin{description}
   \item[I0] For some $\lambda$ there exists a $j:L(V_{\lambda+1})\prec L(V_{\lambda+1}),\text{ with }\crt(j)<\lambda$.
\end{description}

Again, $I0(\lambda)$ expresses what is expected.

Note that if $\lambda$ witnesses I0, then $L(V_{\lambda+1})\nvDash\AC$, because otherwise $L(V_{\lambda+1})\vDash\ZFC$, and we would contradict the proof of Kunen's Theorem \cite{Kunen}, which shows that one cannot have $j:V\prec V$ with critical point less than $\lambda$ and a well-order of $V_{\lambda+1}$ in $V$. The fact that I0 is strictly stronger than I1 was proved by Laver \cite{Laver}.

I0 is probably the most interesting of the rank-into-rank axioms: it is the only very large cardinal that induces a structure on an inner model, therefore creating a new field of research and new tools, and morevoer the structure is reminiscent of the one induced by the Axiom of Determinacy, for reasons that are still not completely understood \cite{Woodin}.
 
 An embedding that witnesses I0 has an ultrapower structure:
 \begin{lem}
 Let $j:L(V_{\lambda+1})\prec L(V_{\lambda+1})$ be such that $\crt(j)<\lambda$. Let
  \begin{equation*}
   U=U_j=\{X\in L(V_{\lambda+1})\cap V_{\lambda+2}:\ j\upharpoonright V_\lambda\in j(X)\}. 
  \end{equation*}
Then $U$ is an $L(V_{\lambda+1})$-ultrafilter such that $\Ult(L(V_{\lambda+1}),\ U)$ is well-founded. By condensation the collapse of $\Ult(L(V_{\lambda+1}),\ U)$ is $L(V_{\lambda+1})$, and $j_U:L(V_{\lambda+1})\prec L(V_{\lambda+1})$, the inverse of the collapse, is an elementary embedding. Moreover, there is an elementary embedding $k_U:L(V_{\lambda+1})\prec L(V_{\lambda+1})$ with $\crt(k_U)>\Theta^{L(V_{\lambda+1})}$ such that $j=k_U\circ j_U$.
\end{lem}
 
 We can use the ultrapower structure to define iterability for $j$:
 \begin{defin}
 Let $j:L(V_{\lambda+1})\prec L(V_{\lambda+1})$ with $\crt(j)<\lambda$ be an elementary embedding, and suppose $j$ is generated by $U=U_j$. Define
 \begin{equation*}
  j(U)=\bigcup\{j(\ran(\pi)): \pi\in L(V_{\lambda+1}), \pi:V_{\lambda+1}\rightarrow U\}
 \end{equation*}
 and then define $j_2$ as the map associated to $j(U)$. 

 Define the successive iterates in the usual way: let $\alpha$ be an ordinal. Then
 \begin{itemize}
  \item if $\alpha=\beta+1$, $M_\beta$ is well-founded and $j_\beta: M_\beta\prec M_\beta$ is the ultrapower via $W$, then $M_\alpha=\Ult(M_\beta, j_\beta(W))$ and $j_\alpha=j_\beta(j_\beta)$.
  \item if $\alpha$ is a limit, let $(M_\alpha, j_\alpha)$ be the direct limit of $(M_\beta, j_\beta)$ with $\beta<\alpha$. 
 \end{itemize}

 We say that $j$ is iterable, if for every $\alpha\in\Ord$, $M_\alpha$ is well-founded and $j_\alpha:M_\alpha\prec M_\alpha$. In this case, we call $j_{\alpha,\beta}$ the natural embeddings between $M_\alpha$ and $M_\beta$.
\end{defin}

The following is a conjunction of Lemma 16 and Lemma 21 in \cite{Woodin}:

\begin{teo}
\label{iterable}
 Let $j:L(V_{\lambda+1})\prec L(V_{\lambda+1})$ with $\crt(j)<\lambda$ be a proper elementary embedding. Then $j$ is iterable. Moreover, for any $n\in\omega$, $j_n:L(V_{\lambda+1})\prec L(V_{\lambda+1})$.
\end{teo}

Theorem \ref{iterable} states that $M_n=L(V_{\lambda+1})$ for $n<\omega$, but $M_\omega$ is definitively different. The key point is that $j_{0,\omega}(\crt(j))=\lambda$, so many characteristics of the critical point of $j$ are transferred by elementarity to $\lambda$ in $M_\omega$. For example, in $L(V_{\lambda+1})$, $\crt(j)$ is measurable and there is a well-ordering of $V_\lambda$, therefore $\lambda$ is measurable in $M_\omega$ and there is a well-ordering of $V_{j_{0,\omega}(\lambda)}^{M_\omega}=V_{j_{0,\omega}(\lambda)}\cap M_\omega$ in $M_\omega$.

Trees are a typical structure that is investigated in combinatorics. Let $\alpha$ be an ordinal. For any $s\in[\alpha]^n$, $\lh(s)=n$. A tree on $\alpha$ is a subset of $[\alpha]^{<\omega}$ closed under initial segments. If $T$ is a tree, for any $s\in T$, denote $T_s=\{t\in T:t\subseteq s\wedge s\subseteq t\}$, $\Suc_T(s)=\{\beta\in\alpha:t^\smallfrown\langle\beta\rangle\in T\}$ and finally for any $n\in\omega$, $\Lev_n(T)=\{s\in T:\lh(s)=n\}$.
 
 \section{General procedure}
  In \cite{Woodin} Woodin introduced Generic Absoluteness for $I0$, while in \cite{DimFried} one of the authors and Sy Friedman used it to prove a single result about the power function and rank-into-rank embeddings. Here we introduce a general procedure that extends the scope of \cite{DimFried} to many more kinds of forcing, and in the next sections we will give some important examples.
  
  One of the most important forcing in dealing with the combinatorics of singular cardinals of cofinality $\omega$ is Prikry forcing. It adds a cofinal sequence to a measurable cardinal.
	\begin{defin}
	 A cardinal $\kappa$ is measurable iff there exists a $\kappa$-complete ultrafilter on $\kappa$.
	\end{defin}
	
  \begin{defin}
   Let $\kappa$ be a measurable cardinal. Fix $U$ a normal measure on $\kappa$. Define $p\in\mathbb{P}$ iff $p=(s,A)$, where $s\in[\kappa]^{<\omega}$, $A\in U$ and $\bigcup s<\bigcap A$. For $p=(s,A),\ q=(t,B)\in\mathbb{P}$, we say $q\leq p$ iff $s\subseteq t$, $B\subseteq A$ and $t\setminus s\subseteq A$.
  \end{defin}
	
	Prikry forcing is useful because it is a very "`delicate"' forcing \cite{Gitik}: it does not add bounded subsets of $\kappa$, and is $\kappa^+$-cc, so it does not change the cardinal structure above $\kappa$. In other words, it makes $\kappa$ singular while changing the universe at least as possible.
  
  The following is instead the tree Prikry forcing:
  \begin{defin}
   Let $\kappa$ be a measurable cardinal. Fix $U$ an ultrafilter on $\kappa$. The tree Prikry forcing $\mathbb{P}$ is the set of conditions $p=(s_p,T^p)$, where $s_p$ is a finite sequence of ordinals in $\kappa$, and $T^p$ is a tree of increasing sequences in $\kappa$ with stem $s_p$, such that for any $t\in T^p$, $\Suc_{T^p}(t)\in U$. We say that $p\leq q$ if $s_p\supseteq s_q$ and $T^p\subseteq T^q$. We say that $p\leq^* q$ if $p\leq q$ and $s_p=s_q$. For any $p\in\mathbb{P}$ and $t\in T^p$, we write $p\oplus t$ for $(t,(T^p)_t)$. 
  \end{defin}

  The difference between the two forcings is minimal: the only difference is that standard Prikry forcing uses a normal ultrafilter, while for tree Prikry forcing normality is not needed. As for the majority of times the ultrafilters are normal, the two forcing are interchangeable, and using one or the other is a matter of better clarity of the proof.
  
  The general technique, in short, will be to start with $I0(\kappa,\lambda)$, and then the following Theorem \ref{general} expresses the fact that if one forces with an ``Easton-like'' forcing and then with a ``Prikry-like'' forcing at $\kappa$, by Generic Absoluteness one can have $I1(\kappa)$. While we have already a formal definition for ``Easton-like'' (reverse Easton iteration), we need a definition of ``Prikry-like'' suitable to our wants. 
  
  \begin{defin}
  \label{geometric}
   Let $\mathbb{P}$ be a forcing notion and $\kappa$ a cardinal. We say that $\mathbb{P}$ is $\kappa$-geometric if
   \begin{itemize}
    \item there exists a length measure of the conditions of $\mathbb{P}$, i.e. $l:\mathbb{P}\to\omega$ such that $l(1_\mathbb{P})=0$ and for any $p,q\in\mathbb{P}$, if $p\leq q$ then $l(p)\geq l(q)$.
    \item for any $\alpha<\kappa$, if $\langle D_\beta:\beta<\alpha\rangle$ is a collection of open dense sets, then for every $p\in\mathbb{P}$ there is a condition $q\leq p$ such that whenever a filter contains $q$ and meets all the dense open sets $E_n=\{p:l(p)>n\}$, it also meets all the $D_\beta$'s.
   \end{itemize}
  \end{defin}

  This notion implies the notion of $\kappa$-goodness that was first introduced by Woodin in \cite{Woodin}, and then perfectioned by Shi in \cite{Shi}. The change is due to the fact that $\kappa$-geometric is more natural in working with tree Prikry-like forcings, even if it works in general:
	
	\begin{lem}
	\label{Prikry}
	 Prikry forcing on $\kappa$ is $\kappa$-geometric.
	\end{lem}
	\begin{proof}
   It is a well known fact that for any $D$ dense set and any $p\in\mathbb{P}$, there exists $q=(s,A)\leq^*p$ and $n\in\omega$ such that for any $t\in [A]^n$, $(q\cup t, A\setminus(\max(t)+1))\in D$, see for example Lemma 1.13 in \cite{Gitik}. Now, suppose $\alpha<\kappa$ and $\langle D_\beta:\beta<\alpha\rangle$ is a collection of open dense sets. Let $\langle q_\beta:\beta\leq\alpha\rangle$ be the sequence built with an iteration of the first sentence, i.e., for any $\beta\leq\alpha$, $q_{\beta+1}\leq^* q_\beta$ is such that there exists an $n_{\beta+1}$ such that for any $t$ of length $n$, $(q\cup t, A\setminus(\max(t)+1))\in D_\beta$, and if $\beta\leq\alpha$ is limit, then by $\kappa$-closeness of $U$ let $q_\beta$ be just $(s, \bigcap_{\gamma<\beta}A_{q_\gamma})$. 

   Therefore $q_\alpha$ is as wanted: let $q'\leq q_\alpha=(s_{q_\alpha},A_{q_\alpha})$ such that $q'\in E_{n_\beta}$. Then 
   \begin{equation*}
    q'<(q_\alpha\cup t, A_{q_\alpha}\setminus(\max(t)+1))\leq (q_\beta\cup t, A_{q_\beta}\setminus(\max(t)+1))\in D_\beta
   \end{equation*}
   for some $t$ with length bigger than $n_\beta$.
  \end{proof}
	
  The key point of the lemma, and in fact of any proof of $\kappa$-geometricness, is the variation of the Prikry condition\footnote{The Prikry condition is: for every $p\in\mathbb{P}$ and for every $\sigma$ statement in the forcing language, there exists $q\leq^*p$ such that $q\vDash\sigma$ or $q\vDash\neq\sigma$.} that is presented in the first line of the proof. We isolate it:
  
  \begin{defin}
   Let $(\mathbb{P},\leq,\leq^*)$ be a forcing notion with a length measure (as in Definition \ref{geometric}) such that $p\leq^* q$ iff $l(p)=l(q)$. Then $\mathbb{P}$ satisfies the *-Prikry condition iff for every $p\in\mathbb{P}$ and for every dense $D\subseteq\mathbb{P}$, there are $n\in\omega$ and $q\leq^*p$ such that for any $r\leq q$ with $l(r)=l(q)+n$, $r\in D$.
  \end{defin}

  Such variation is pretty common in literature, even if it is often not explicitly stated, as the proof basically repeats the proof of the Prikry condition. It goes back to Prikry (\cite{Prikry}) and Mathias (\cite{Mathias}). For completeness, we will sketch the proof for the tree Prikry forcing, so it will be clear to the reader how the proof goes for other kinds of forcing that satisfy the Prikry condition.

  \begin{lem}
   \label{*tree}
   Let $\kappa$ be a measurable cardinal. Then the tree Prikry forcing $\mathbb{P}$ on $\kappa$ satisfies the *-Prikry condition.
  \end{lem}
  \begin{proof}[Sketch of the proof]
   Fix $D$ dense and a condition $p\in\mathbb{P}$ with tree $T^p$; we should find a way to shrink $T^p$ to a $T^q$ so that $q=(s_p,T^q)$ is such that any $n$-extension of $q$ is in $D$.
  
   The first step is to find a $T'\subseteq T^p$ such that for any $s\in T^p$ if there exists a $q\leq^*(s,T'_s)$, $q\in D$, then $(s,T'_s)\in D$, that is, such that the tree-part has no role in establishing whether extensions of $p$ are in $D$ or not. This can be done by induction on levels, choosing for any $s$ a $T$ such that $(s,T)\in D$, when it exists, and intersecting everything. The final $T'$ will be such that $(s_p,T')\in\mathbb{P}$ by completeness of the measure.
   
   The second step is to reduce the tree again to a $T^q$ such that if $(s,(T^p)_s)\leq p$ and $(s,(T^p)_s)\in D$, then all the extensions of $p$ of the same length are in $D$. This exploits the fact that the successor of any element of $T'$ are of measure one, therefore, as an example, consider $S=\Suc_T'(s_p)$. Then either the elements of $S$ that are ``in $D$'' (remember that by the first construction ``being in $D$'' does not depend on the tree-part) form a measure one set, or those that are not form a measure one set, and we cut the branches that are not in such set. By induction, we do this on all levels, also extending to double successors (i.e. we consider the elements such that all their successors have all their successors either in $D$ or not), triple successors, etc... Again by completeness the resulting $T^q$ is such that $(s_p,T^q)\in\mathbb{P}$.
   
   Then we have $(s_p, T^q)$ such that all its successors of length $n$, for every $n$, are either in $D$ or not. If there exists an $n$ such that all the $n$-successors are in $D$, then we are done. But there must be one, because $D$ is dense.
  \end{proof}

  \begin{cor}
 	\label{tree} 
    Let $\kappa$ be a measurable cardinal. Then the tree Prikry forcing $\mathbb{P}$ on $\kappa$ is $\kappa$-geometric.
  \end{cor}

 \begin{defin}
\label{Definition}
 Let $\mathbb{P}_\lambda$ be a forcing iteration of length $\lambda$, where $\lambda$ is either a strong limit cardinal or is equal to $\infty$, the class of all ordinals. We say that $\mathbb{P}_\lambda$ is
 \begin{itemize}
  \item \emph{reverse Easton} if nontrivial forcing is done only at infinite cardinal stages, direct limits are taken at all inaccessible cardinal limit stages, and inverse limits are taken at all other limit stages; moreover, $\mathbb{P}_\lambda$ is the direct limit of the $\langle\mathbb{P}_\delta$, $\delta<\lambda\rangle$ if $\lambda$ is regular or $\infty$, the inverse limit of the $\langle\mathbb{P}_\delta$, $\delta<\lambda\rangle$, otherwise;
  \item \emph{directed closed} if for all $\delta<\lambda$, $\mathbb{Q}_\delta$ is $<\delta$-directed closed, i.e., for any $D\subseteq\mathbb{Q}_\delta$, $|D|<\delta$ such that for any $d_1,d_2\in D$ there is an $e\in D$ with $e\leq d_1, e\leq d_2$, there exists $p\in\mathbb{Q}_\delta$ such that $p\leq d$ for any $d\in D$;
  \item \emph{$\lambda$-bounded} if for all $\delta<\lambda$, $\mathbb{Q}_\delta$ has size $<\lambda$. Note that in the case $\lambda=\infty$, this just means that each $\mathbb{Q}_\delta$ is a set-forcing;
 \end{itemize}
 Moreover, if $j$ is any elementary embedding such that $j''\lambda\subset\lambda$ and $\mathbb{P}_\lambda\subset\dom(j)$, we say that $\mathbb{P}_\lambda$ is $j$-coherent if for any $\delta<\lambda$, $j(\mathbb{P}_\delta)=\mathbb{P}_{j(\delta)}$.
\end{defin}
  
  The following theorem summarizes the general procedure:
  \begin{teo}
  \label{general}
   Let $j:L(V_{\lambda+1})\prec L(V_{\lambda+1})$ with $\crt(j)=\kappa<\lambda$. Let $\mathbb{P}$ be a directed closed, $\lambda$-bounded, $j$-coherent reverse Easton iteration. Let $\mathbb{Q}$ be a $\kappa$-geometric forcing in $(V_\lambda)^{V^{\mathbb{P}}}$ that adds a Prikry sequence to $\kappa$. Then there exist $G$ generic for $\mathbb{P}$ and $H$ $V[G]$-generic for $\mathbb{Q}$ such that $V[G][H]\vDash\exists k:V_{\kappa+1}\prec V_{\kappa+1}$. 
  \end{teo}

  The relevant point of the proof is the forcing $\mathbb{Q}$, as by Lemma 3.6 and Lemma 3.7 in \cite{DimFried} combined, for any $G$ generic for $\mathbb{P}$, $V[G]\vDash\exists j:L(V_{\lambda+1})\prec L(V_{\lambda+1}),\ \crt(j)=\kappa$. So, for better readability, from now on we call the generic extension of $\mathbb{P}$ just $V$.
	
	Let $j_{0,\omega}:L(V_{\lambda+1})\prec M_\omega$ the $\omega$-th iterate of $j$. Then $j_{0,\omega}(\mathbb{Q})$ is a $\lambda$-geometric forcing that adds a Prikry sequence to $\lambda$ in $M_\omega$.
	
  \begin{lem}
   In $V$ there are only $\lambda$ open dense sets of $j_{0,\omega}(\mathbb{Q})$.
  \end{lem}
  \begin{proof}
   As $\mathbb{Q}\in V_\lambda$, there exists $n\in\omega$ such that $\mathbb{Q}\in V_{\kappa_n}$. In particular $j_{0,\omega}(\mathbb{Q})\in M_\omega\cap V_{j_{0,\omega}(\kappa_n)}$ and its dense sets are in $j_{0,\omega}(\mathbb{Q})\in M_\omega\cap V_{j_{0,\omega}(\kappa_{n+1})}$. But $|M_\omega\cap V_{j_{0,\omega}(\lambda)}|=\lambda$. 
  \end{proof}

  \begin{prop}
   There exists a generic ultrafilter $H\in V$ of $j_{0,\omega}(\mathbb{Q})$.
  \end{prop}
  \begin{proof}
   Let $\langle D_\alpha:\alpha<\lambda\rangle$ be an enumeration of the dense sets of $j_{0,\omega}(\mathbb{Q})$ in $V$. For every $n\in\omega$, fix $q_n$ that witnesses $\lambda$-geometricness for $\langle D_\alpha:\alpha<\kappa_n\rangle$. Then for every $m$ there exists a $q'_{n,m}<q_n$ such that $q'_{n,m}\in E_m$. Let $H$ be the filter $\bigcup_{n,m\in\omega}{\cal F}_{q'_{n,m}}$, with ${\cal F}_q$ the filter generated by $q$. Then $H$ is generic.
  \end{proof}

  The following appeared in \cite{Woodin} as Theorem 136.

  \begin{teo}[Generic Absoluteness]
  \label{GenericAbsoluteness}
   Let $j:L(V_{\lambda+1})\prec L(V_{\lambda+1})$ with $\crt(j)<\lambda$ be a proper elementary embedding. Let $(M_\omega,j_\omega)$ be the $\omega$-th iterate of $j$ and let $\langle\eta_i:i<\omega\rangle\in V$ be a sequence generic for the Prikry forcing on $\lambda$ in $M_\omega$. 

   Then for all $\alpha<\lambda$ there exists an elementary embedding 
   \begin{equation*}
    \pi:L_\alpha(M_\omega[\langle\eta_n:n\in\omega\rangle]\cap V_{\lambda+1})\prec L_\alpha(V_{\lambda+1})
   \end{equation*}
   such that $\pi\upharpoonright\lambda$ is the identity. 
  \end{teo}

  In particular, as I1($\lambda$) holds in $L_1(V_{\lambda+1})$, $M_\omega[\langle\eta_n:n\in\omega\rangle]\vDash I1(\lambda)$.

  \begin{proof}[Proof of Theorem \ref{general}]
	 Let $H\in V$ be $j_{0,\omega}(\mathbb{P})$-generic. Fix $g$, one of the Prikry sequences added by $j_{0,\omega}(\mathbb{Q})$. Then $M_\omega\subseteq M_\omega[g]\subseteq M_\omega[H]\subseteq V$. But $M_\omega[g]$ satisfies Generic absoluteness' conditions, therefore $M_\omega[g]\vDash I1(\lambda)$. But also $V\vDash I1(\lambda)$, therefore it must be $M_\omega[H]\vDash I1(\lambda)$: As, by \ref{GenericAbsoluteness}, $M_\omega[g]\cap V_{\lambda+1}\prec V_{\lambda+1}$, it is immediate to see that for any formula $\varphi$ and for any $a\in M_\omega[g]$, $M_\omega[g]\cap V_{\lambda+1}\vDash\varphi(a)$ iff $M_\omega[H]\cap V_{\lambda+1}\vDash\varphi(a)$ iff $V_{\lambda+1}\vDash\varphi(a)$. But this is the situation of Remark \ref{I1}, therefore $M_\omega[H]\vDash I1(\lambda)$.
		
	 We just proved that 
	 \begin{equation*}
    M_\omega\vDash\exists p\in j_{0,\omega}(\mathbb{Q})\ p\Vdash_{j_{0,\omega}(\mathbb{Q})}\exists i:(V_{\check{\lambda}+1})\prec (V_{\check{\lambda}+1}) ,
   \end{equation*}
   Applying $j^{-1}$, we have that 
   \begin{equation*}
    L(V_{\lambda+1})\vDash\exists p\in\mathbb{Q}\ p\Vdash_{\mathbb{Q}}\exists i:(V_{\check{\kappa}+1})\prec (V_{\check{\kappa}+1}),
   \end{equation*}
   as we wanted to prove.
  \end{proof}

 \section{Extender-based Prikry forcing}
  The first application of $\kappa$-geometricness will be on the extender-based Prikry forcing. It was introduced by Gitik and Magidor, and the reader can find an exhaustive description in \cite{Gitik}. The aim of the forcing is to add many Prikry sequences to a strong enough cardinal, blowing up its power while not changing the power function below it. This is more difficult than just having $\lambda$ singular and $2^\lambda>\lambda^+$: the proof for this is to take $\lambda$ measurable, forcing $2^\lambda>\lambda^+$ and then adding a Prikry sequence to $\lambda$. But Dana Scott \cite{Scott} proved that if $\lambda$ is measurable and $2^\lambda>\lambda^+$, then for a measure one set below $\lambda$, $2^\kappa>\kappa^+$, therefore this method would not give the first failure of GCH on $\lambda$. The solution is to exploit the extender structure of the cardinal to add many Prikry sequences, at the same time blowing up the power and changing the cofinality.
  
  \begin{defin}
   Let $\kappa$ and $\gamma$ be cardinal. Then $\kappa$ is $\gamma$-strong iff there is a $j:V\prec M$ such that $\crt(j)=\kappa$, $\gamma<j(\kappa)$ and $V_{\kappa+\gamma}\subseteq M$.
  \end{defin}
  
  We write the definition as it is in \cite{Gitik}.

  Suppose \GCH, and let $\lambda$ be a $2$-strong cardinal.
  
  For any $\alpha<\lambda^{++}$, define a $\lambda$-complete normal ultrafilter on $\lambda$ as $X\in U_\alpha$ iff $\alpha\in j(X)$. For any $\alpha,\beta<\lambda^{++}$, define $\alpha\leq_E\beta$ iff $\alpha\leq\beta$ and for some $f\in^\lambda\lambda$, $j(f)(\beta)=\alpha$. Then by a result in \cite{Gitik}, $\langle\lambda^{++},\leq\rangle$ is a $\lambda^{++}$-directed order, and there exists $\langle\pi_{\alpha\beta}:\alpha,\beta\in\lambda^{++},\alpha\leq_E\beta\rangle$ such that $\langle\lambda^{++}, \langle U_\alpha:\alpha<\lambda^{++}\rangle, \leq_E\rangle$ is a nice system. There is no need to define a nice system here, the term is introduced only because the extender-based Prikry forcing is built on a nice system, the full definition can be found in \cite{Gitik}.

  Fix a nice system $\langle\lambda^{++}, \langle U_\alpha:\alpha<\lambda^{++}\rangle, \leq_E\rangle$. For any $\nu<\lambda$ and $\lambda<\alpha<\lambda^{++}$, let us denote $\pi_{\alpha,0}(\nu)$ by $\nu^{\alpha,0}$. We will write just $\nu^0$ when $\alpha$ is obvious. By a \textdegree-increasing sequence of ordinals we mean a sequence $\langle\nu_0,\dots,\nu_n\rangle$ of ordinals below $\lambda$ such that $\nu_0^0<\dots<\nu_n^0$. We say that $\mu$ is permitted for $\langle\nu_0,\dots,\nu_n\rangle$ iff  $\mu^0>\nu_i^0$ for all $i=0\dots n$. 

  \begin{defin}
  \label{P}
   The set of forcing conditions $\mathbb{P}$ consists of all the elements $p$ of the form 
   \begin{equation*}
    \{\langle\gamma,p^{\gamma}\rangle | \gamma\in g\setminus\{\max(g)\}\}\cup\{\langle\max(g),p^{\max(g)},T\rangle\},
   \end{equation*}
   where
   \begin{enumerate}
    \item $g\subseteq\lambda^{++}$ of cardinality $\leq\lambda$ which has a maximal element according to $\leq_E$ and $0\in g$. 
	  \item for $\gamma\in g$, $p^\gamma$ is a finite \textdegree-increasing sequence of ordinals $<\lambda$.
	  \item $T$ is a tree, with a trunk $p^{\max(g)}$, consisting of \textdegree-increasing sequences. All the splittings in $T$ are required to be on sets in $U_{\max(g)}$, i.e., for every $\eta\in T$, if $\eta\geq p^{\max(g)}$ then the set 
     \begin{equation*}
      \Suc_T(\eta)=\{\mu<\lambda:\eta^\smallfrown\langle\mu\rangle\in T\}\in U_{\max(g)}.
     \end{equation*}
	   Also require that for $\eta_1\geq_T\eta_2\geq_T p^{mc}$, $\Suc_T(\eta_1)\subseteq\Suc_T(\eta_2)$.
	  \item For every $\mu\in\Suc_T(p^{\max(g)})$, $|\{\gamma\in g:\mu$ is permitted for $p^{\gamma}\}|\leq\mu^0$.
	  \item For every $\gamma\in g$, $\pi_{\max(g),\gamma}(\max(p^{\max(g)}))$ is not permitted for $p^\gamma$.
	  \item $\pi_{\max(g),0}$ projects $p^{\max(g)}$ onto $p^0$ (so $p^{\max(g)}$ and $p^0$ are of the same length).
   \end{enumerate}
  \end{defin}

  Let us denote $g$ by $\supp(p)$, $\max(g)$ by $\mc(p)$, $T$ by $T^p$, $p^{\max(g)}$ by $p^{\mc}$ and $\bas(p)=p\upharpoonright(\supp(p)\setminus\mc(p))$.

  \begin{defin}
  \label{less}
   Let $p,q\in\mathbb{P}$. We say that $p$ extends $q$ and denote this by $p<q$ iff
	 \begin{enumerate}
	  \item $\supp(p)\supseteq\supp(q)$.
	  \item For every $\gamma\in\supp(q)$, $p^\gamma$ is an end-extension of $q^\gamma$.
	  \item $p^{\mc(q)}\in T^q$.
	  \item For every $\gamma\in\supp(q)$, 
     \begin{equation*}
      p^\gamma\setminus q^\gamma=\pi_{\mc(q),\gamma}[(p^{\mc(q)}\setminus q^{\mc(q)})\upharpoonright(\lh(p^{\mc})\setminus (i+1))], 
     \end{equation*}
     where $i\in\dom(p^{\mc(q)})$ is the largest such that $p^{\mc(q)}(i)$ is not permitted for $q^\gamma$.
	  \item $\pi_{\mc(p),\mc(q)}$ projects $T^p_{p^{\mc}}$ into $T^q_{q^{\mc}}$.
	  \item For every $\gamma\in\supp(q)$ and $\mu\in\Suc_{T^p}(p^{\mc})$, if $\mu$ is permitted for $p^\gamma$, then $\pi_{\mc(p),\gamma}(\mu)=\pi_{\mc(q),\gamma}(\pi_{\mc(p),\mc(q)}(\mu))$.
	 \end{enumerate}
  \end{defin}

  \begin{defin}
   Let $p,q\in\mathbb{P}$. We say that $p$ is a direct extension of $q$ and denot this by $p<^* q$ iff
	 \begin{enumerate}
	  \item $p<q$
	  \item for every $\gamma\in\supp(q)$, $p^\gamma=q^\gamma$.
	 \end{enumerate}
  \end{defin}
  
  \begin{lem}[\cite{Gitik}]
   $\leq^*$ is $\lambda$-complete.
  \end{lem}

  We need to define what an $n$-extension is for this forcing:
  
  \begin{defin}
   Let $p\in\mathbb{P}$ and $t\in T^p_{p^{\mc}}$. Then $p\oplus t$ is defined as follows:
	 \begin{enumerate}
	  \item $\supp(p\oplus t)=\supp(p)$;
	  \item $(p\oplus t)^{\mc}={p^{\mc}}^\smallfrown t$;
	  \item $T^{p\oplus t}=\{s\in T^p: s\subseteq (p\oplus t)^{\mc}\vee (p\oplus t)^{\mc}\subseteq s\}$;
	  \item if $\gamma\in\supp(p)$, 
     \begin{equation*}
      (p\oplus t)^\gamma={p^{\gamma}}^\smallfrown\pi_{\mc(p),\gamma}[t\upharpoonright(\lh(t)\setminus (i_\gamma+1))], 
     \end{equation*}
    where $i_\gamma$ is the largest such that $t(i)$ is not permitted by $p^\gamma$.
   \end{enumerate}
  \end{defin}

  A condition in $\mathbb{P}$ is therefore a set of finite sequences and $T$ indicating the possible extensions not only of the last one, but, via projection, of all of them. Morally speaking, $p\oplus t$ is the largest extension of $p$ that we can have choosing $t$ (and its projections) as extension. 

  \begin{teo}[Gitik, Magidor]
   Let $\mathbb{P}$ as above. Then 
   \begin{equation*}
    V^{\mathbb{P}}\vDash 2^\lambda=\lambda^{++}\ \wedge\ \forall\kappa<\lambda\ 2^\kappa=\kappa^+
   \end{equation*}
  \end{teo} 

  In spirit, Gitik-Magidor extender Prikry forcing is a tree Prikry forcing, therefore the proof for the *-Prikry condition is, in spirit, the same of Lemma \ref{*tree}. In practice, it is much more complex, and the proof of that goes back to Merimovich (\cite{Merimovich}). In this case, it has this form: for every $p\in\mathbb{P}$ and $D\subseteq\mathbb{P}$ dense there exist $q\leq^+p$ and $n\in\omega$ such that for any $t\in T^q$, $\lh(t)=n$, $q\oplus t\in D$.
  
  It is tempting then to do as Lemma \ref{Prikry}, i.e., building a descending sequence of conditions that satisfy the *-Prikry condition. But there is a technical difference: in \ref{Prikry} (using the same notation), a $n_\beta$-extension of $q_\alpha$ was also a $n_\beta$ extension of $q_\beta$, and therefore in $D_\beta$; in this case, a $n_\beta$-extension of $q_\alpha$ can possibly not be a $n_\beta$ extension of $q_\beta$, for example when $q_\beta^{mc}\neq q_\alpha^{mc}$. Therefore we have to check that everything projects smoothly.
  
  \begin{prop}
   Let $\mathbb{P}$ as above. Then $\mathbb{P}$ is $\lambda$-geometric.
  \end{prop} 
  \begin{proof}
   Let $\langle D_\beta:\beta<\alpha\rangle$ be a sequence of open dense sets, with $\alpha<\lambda$, $p\in\mathbb{P}$, and let $\langle q_\beta:\beta\leq\alpha\rangle$, with $q_0=p$, be the sequence built with an iteration of the *-Prikry condition, i.e., for any $\beta\leq\alpha$, $q_{\beta+1}\leq^* q_\beta$ is such that there exists an $n_{\beta+1}$ such that for any $t\in T^{q_\beta}$ of length $n_{\beta+1}$, $q_{\beta+1}\oplus t\in D_\beta$, and if $\beta\leq\alpha$ is limit, then by $\lambda$-closeness of $\leq^*$ let $q_\beta$ be such that $q_\beta\leq^* q_\gamma$ for all $\gamma<\beta$. We can assume $p=\{\langle 0, \langle\rangle, TC\}$, where TC is the complete tree of the increasing finite sequences in $\lambda$. 
     \begin{claim}
      For any $\beta<\alpha$, for any $t\in T^{q^\alpha}$, $q_\alpha\oplus t\leq^* q_\beta\oplus \pi_{\mc(q_\alpha),\mc(q_\beta)}''t$.
     \end{claim}
     \begin{proof}
      Note that $q_\alpha^{\mc}=q_\beta^{\mc}=\langle\rangle$\footnote{The claim is in fact true in general, but in this case calculation is easier}. We prove it point by point.
      \begin{itemize}
       \item $q_\beta\oplus\pi_{\mc(q_\alpha),\mc(q_\beta)}''t$ is well defined, i.e., $\pi_{\mc(q_\alpha),\mc(q_\beta)}''t\in T^{q_\beta}$. This is true for Definition \ref{less}(5).
       \item $\supp(q_\alpha\oplus t)\supseteq\supp(q_\beta\oplus\pi_{\mc(q_\alpha),\mc(q_\beta)}''t)$. This is true because $\supp(q_\alpha\oplus t)=\supp(q_\alpha)$, 
        \begin{equation*}
         \supp(q_\beta\oplus\pi_{\mc(q_\alpha),\mc(q_\beta)}''t)=\supp(q_\beta) 
        \end{equation*}
        and by Definition \ref{less}(1).
       \item for any $\gamma\in\supp(q_\alpha)$, $(q_\alpha\oplus t)^\gamma=(q_\beta\oplus\pi_{\mc(q_\alpha),\mc(q_\beta)}''t)^\gamma$. By definition, and since $(q_\alpha)^\gamma=(q_\beta)^\gamma$, this is true if and only if 
        \begin{equation*}
         \pi_{\mc(q_\alpha),\gamma}[t\upharpoonright(\lh(t)\setminus(i_\gamma+1))]=\pi_{\mc(q_\beta),\gamma}[\pi_{\mc(q_\alpha),\mc(q_\beta)}''t\upharpoonright (\lh(t)\setminus (j_\gamma+1))], 
        \end{equation*}
        where $i_\gamma$ is the largest such that $t(i_\gamma)$ is not permitted by $(q_\alpha)^\gamma$ and $j_\gamma$ is the largest such that $\pi_{\mc(q_\alpha),\mc(q_\beta)}''t(j_\gamma)$ is not permitted by $(q_\beta)^\gamma=(q_\alpha)^\gamma$. By Definition \ref{P}(6) $(q_\alpha)^0=(q_\beta)^0=\langle\rangle$, therefore by Definition \ref{less}(6) 
        \begin{equation*}
         \pi_{\mc(q_\alpha),0}\upharpoonright t=\pi_{\mc(q_\beta),0}\circ\pi_{\mc(q_\alpha),\mc(q_\beta)}\upharpoonright T, 
        \end{equation*}
        so $i_\gamma=j_\gamma$. Therefore this point is true by Definition \ref{less}(6).
       \item $(q_\alpha\oplus t)^{\mc(q_\beta\oplus\pi_{\mc(q_\alpha),\mc(q_\beta)}''t)}\in T^{q_\beta\oplus\pi_{\mc(q_\alpha),\mc(q_\beta)}''t}$. First of all, $\mc(q_\beta\oplus\pi_{\mc(q_\alpha),\mc(q_\beta)}''t)=\mc(q_\beta)$. By definition, 
        \begin{equation*}
         (q_\alpha\oplus t)^{\mc(q_\beta)}=(q_\alpha)^{\mc(q_\beta)} {^\smallfrown}\pi_{\mc(q_\alpha),\mc(q_\beta)}[t\upharpoonright(\lh(t)\setminus (i_{\mc(q_\beta)}+1))], 
        \end{equation*}
        with $i_{\mc(q_\beta)}$ as above. By Definition \ref{less}(2) $(q_\alpha)^{\mc(q_\beta)}=q_\beta^{\mc}=\langle\rangle$, and therefore $i_{\mc(q_\beta)}=0$, so the point follows simply by definition of $T^{q_\beta\oplus\pi_{\mc(q_\alpha),\mc(q_\beta)}''t}$.
       \item $\pi_{\mc(q_\alpha),\mc(q_\beta)}$ projects $T_t^{q_\alpha\oplus t}$ into $T_{\pi_{\mc(q_\alpha),\mc(q_\beta)}''t}^{q_\beta\oplus \pi_{\mc(q_\alpha),\mc(q_\beta)}''t}$. By definition, $T_t^{q_\alpha\oplus t}= T_t^{q_\alpha}$ and $T_{\pi_{\mc(q_\alpha),\mc(q_\beta)}''t}^{q_\beta\oplus \pi_{\mc(q_\alpha),\mc(q_\beta)}''t}=T_{\pi_{\mc(q_\alpha),\mc(q_\beta)}''t}^{q_\beta}$, so this is true by Definition \ref{less}(5).
       \item for any $\gamma\in\supp(q_\alpha)$ and $\mu\in\Suc_{T^{q\oplus t}}(t)$, if $\mu$ is permitted for $(q_\alpha\oplus t)^\gamma$ then 
        \begin{equation*}
         \pi_{\mc(q_\alpha),\gamma}(\mu)=\pi_{\mc(q_\beta),\gamma}(\pi_{\mc(q_\alpha),\mc(q_\beta)}(\mu)). 
        \end{equation*}
        As $T^{q_\alpha\oplus t}$ is a subtree of $T^{q_\alpha}$, and $\mu$ is permitted for $(q_\alpha\oplus t)^\gamma$ means that it is also permitted for $(q_\alpha)^\gamma$, this is a direct consequence of Definition \ref{less}(6). 
      \end{itemize}
     \end{proof}

     Therefore $q_\alpha$ is as wanted: let $q'\leq q_\alpha$ such that $q'\in E_{n_\beta}$. Then $q'<q_\alpha\oplus t$ for some $t\in T^{q_\alpha}$ with length $n_\beta$. But by the Claim $q_\alpha\oplus t \leq^* q_\beta\oplus \pi_{\mc(q_\alpha),\mc(q_\beta)}''t$, and by definition $q_\beta\oplus \pi_{\mc(q_\alpha),\mc(q_\beta)}''t\in D_\beta$, so also $q'\in D_\beta$.
   \end{proof}

  \begin{cor}
  \label{nonsch}
   Suppose $I0(\kappa,\lambda)$. Then there exists a generic extension in which $I1(\kappa)+2^\kappa=\kappa^{++}+\forall\eta<\kappa\ 2^\eta=\eta^+$.
  \end{cor}
  \begin{proof}
   We want to apply Theorem \ref{general}. so first note that the forcing that forces \GCH{} below $\lambda$ is a directed closed, $\lambda$-bounded, $j$-coherent reverse Easton iteration, while the forcing that forces \GCH{} above $\lambda$ is $\lambda$-closed, and therefore does not touch $I0$. $I0$ clearly implies I2, and it is a well-known fact (see e.g. Proposition 24.2 in \cite{Kanamori}) that this is equivalent to the existence of $k:V\prec M$ with $V_\lambda\subseteq M$. We can construct $k$ so that $k\upharpoonright V_\lambda=j\upharpoonright V_\lambda$, so $\kappa$ is $2$-strong. Therefore we can apply the extender-based Pikry forcing to $\kappa$. The elements of $\mathbb{P}$ on $\kappa$ are $\kappa$-sequences of triples of elements of $\kappa^{++}$, finite sequences in $\kappa$ and functions from $\kappa^\omega$ to ${\cal P}(\kappa)$, so we can say that the forcing is in $V_{\kappa_1}$. The forcing adds a Prikry sequence to $\kappa$ and it 
is $\kappa$-geometric, therefore the conditions of Theorem \ref{general} are met.
  \end{proof}
 
\section{Diagonal Supercompact Prikry forcings}

There are many versions of the diagonal supercompact Prikry forcing, we are going to use the one in \cite{GitikSharon} (and later the one in \cite{Neeman}). First, there is a preparation forcing that forces $2^\alpha=\alpha^{+\omega+2}$ for all $\alpha$ inaccessible. Then, the diagonal supercompact forcing exploits the fine normal ultrafilters that come from enough supercompactness of a cardinal to add Prikry sequences to it, while inducing an interesting pcf structure.

\begin{defin}
 Let $\kappa,\gamma$ be cardinals. We say that $\kappa$ is $\gamma$-supercompact iff there exists a fine normal measure on ${\cal P}_\kappa(\gamma)$, i.e., a measure $U$ such that for any $f:{\cal P}_\kappa(\eta)\to \gamma$ such that $f(x)\in x$ for almost every $x$, $f$ is constant on a set in $U$.
\end{defin}

One interesting combinatorial principle is $\Box_\kappa$, that states the existence of a coherent collection of clubs. While many combinatorial principles are consistent with the existence of large cardinals, $\Box_\kappa$ fails above large enough cardinals (Solovay). It is of interest, therefore, investigating weakings of such principle.

\begin{defin}
 We say that a cardinal $\kappa$ has the approachability property, $AP_\lambda$, iff there exists a sequence $\langle C_\alpha:\alpha<\kappa^+\rangle$ such that
 \begin{itemize}
  \item for $\alpha$ limit, $C_\alpha$ is a club in $\alpha$ and $\ot(C_\alpha)=\cof(\alpha)$;
  \item there is a club $D\subseteq\kappa^+$ such that for any $\alpha\in D$, for any $\beta<\alpha$ there exists $\gamma<\alpha$ such that $C_\alpha\cap\beta=C_\gamma$.
 \end{itemize}
\end{defin}

It is not difficult to see that it is a weakening of $\Box_\kappa$.

Another field of research in infinite combinatorics is pcf theory: given a cardinal $\kappa$ and $\langle\mu_n:n\in\omega\rangle$ cofinal, it investigates the structure of the functions in $\Pi_{n\in\omega}\mu_n$, and it is a standard tool for the analysis of the combinatorics of a cardinal of cofinality $\omega$.  

\begin{defin}
 Let $\langle\mu_n:n\in\omega\rangle$ be a sequence cofinal in $\kappa$. A sequence $\langle f_\alpha:\alpha<\kappa^+\rangle\subseteq\Pi_{n\in\omega}\mu_n$ is a very good scale iff
 \begin{itemize}
  \item $\langle f_\alpha:\alpha<\kappa^+\rangle$ is a scale, i.e. such that for every $\alpha<\beta<\kappa^+$, $f_\alpha(m)<f_\beta(m)$ for almost every $m$ and for every $f\in\Pi_{n\in\omega}\mu_n$ there exists $\beta<\kappa^+$ and $n\in\omega$ with $f(m)<f_\beta(m)$ for every $m>n$. 
  \item for every $\beta<\kappa^+$ such that $\omega<\cof(\beta)$ there exists a club $C$ of $\beta$ and $n<\omega$ such that $f_{\gamma_1}(m)<f_{\gamma_2}(m)$ for every $\gamma_1<\gamma_2\in C$ and $m>n$.
 \end{itemize}
 If $\kappa$ is as above, we say that there exists a very good scale in $\kappa$, $VGS_\kappa$.
\end{defin}

Both these properties don't hold above a supercompact cardinal, and in \cite{GitikSharon} it is proven that having a very good scale does not imply the approachability property. We will prove that this holds also under rank-into-rank hypotheses. 

Let $\mathbb{E}$ be the reverse Easton forcing of length $\lambda$ that force $2^\alpha=\alpha^{+\omega+2}$ for all $\alpha$ inaccessible. This forcing is:
 \begin{itemize}
  \item directed closed: $\mathbb{Q}_\alpha$, the forcing that adds $\alpha^{+\omega+2}$ subsets of $\alpha$, is $<\alpha$-directed closed;
  \item $\lambda$-bounded: as $\lambda$ is strong limit $|\mathbb{Q}_\alpha|<\lambda$. 
  \item $j$-coherent: as $j(\mathbb{Q}_\alpha)$ is the poset consisting of the functions whose domain is a subset of $j(\alpha)^{+\omega+2}$ of size less than $j(\alpha)$, and whose range is a subset of the partial functions between $j(\alpha)$ and $j(\alpha)$, that is, $j(\mathbb{Q}_\alpha)=\mathbb{Q}_{j(\alpha)}$.
 \end{itemize}

Let $\kappa$ be $\kappa^{+\omega+2}$-supercompact, with $U_\omega$ witnessing it and let $U_n$ be the projection of $U_\omega$ on ${\cal P}_\kappa(\kappa^{+n})$, i.e., 
\begin{equation*}
 X\in U_n\text{ iff }\{P\in{\cal P}_\kappa(\kappa^{+\omega+2}):P\cap\kappa^{+n}\in X\}\in U_\omega. 
\end{equation*}
 Clearly $U_n$ is a normal ultrafilter on ${\cal P}_\kappa(\kappa^{+n})$. 

Let $a,b\in{\cal P}_\kappa(\kappa^{+n})$ and $b\cap\kappa\in\kappa$. Set
 \begin{equation*}
  a\underset{\sim}{\subset}b \leftrightarrow a\subseteq b \wedge \ot(a)<b\cap\kappa.
 \end{equation*}
 
 \begin{defin}
 \label{diagonalprikry}
  $p\in\mathbb{Q}$ iff $p=\langle a_0^p,a_1^p,\dots,a_{n-1}^p,X_n^p,X_{n+1}^p,\dots\rangle$ where
  \begin{enumerate}
   \item $\forall l<n\ a_l^p\in{\cal P}_\kappa(\kappa^{+n})$ and $a_l^p\cap\kappa$ is an inaccessible cardinal;
   \item $\forall m\geq n$, $X_m^p\in U_m$;
   \item $\forall m\geq n\ \forall b\in X_m^p\ \forall l<n\ a_l^p\underset{\sim}{\subset} b$;
   \item $\forall i<j<l$ $a_i^p\underset{\sim}{\subset} a_j^p$.
  \end{enumerate}
 \end{defin}

 For $p=\langle a_0^p,a_1^p,\dots,a_{n-1}^p,X_n^p,X_{n+1}^p,\dots\rangle$, let us denote $n$ as $l(p)$. Moreover, for any collection of $A_i\subseteq{\cal P}_\kappa(\kappa^{+i})$, let
 \begin{equation*}
  \tilde{\Pi}_{n\in m}A_n=\{\langle a_0,\dots, a_{m-1}\rangle:\forall i<j<m\ a_i\in A_i\wedge a_i \underset{\sim}{\subset}a_j\}.
 \end{equation*}
 
 For any collection of $A_a$, $a\in{\cal P}_\kappa(\kappa^{+n})$, let 
 \begin{equation*}
  \Delta A_a=\{b\in{\cal P}_\kappa(\kappa^{+n}):\forall a\in{\cal P}_\kappa(\kappa^{+n})\ a\underset{\sim}{\subset} b\rightarrow b\in X_a\}. 
 \end{equation*}
 It is a standard result that if each $A_a\in U_n$, then $\Delta A_a\in U_n$.
 
 \begin{defin}
  Let $p,q\in\mathbb{Q}$. Then $p\leq^*q$ iff
  \begin{enumerate}
   \item $l(p)=l(q)$;
   \item $\forall l<l(p)\ a_l^p=a_l^q$;
   \item $\forall m\geq l(p)\ X_m^p\subseteq X_m^q$.
  \end{enumerate}
 \end{defin}

 \begin{defin}
  Let $p\in\mathbb{Q}$ and $\vec{a}\in\tilde{\Pi}_{l(p)\leq n\leq m}X_n^p$. Then we denote by $p\oplus\vec{a}$ the sequence $\langle a_0^p,\dots a_{l(p)}^p,a(l(p)),\dots,a(m), Y_{m+1}^p,\dots\rangle$, where
  \begin{equation*}
   Y_n=\{b\in X_n^p:\forall l(p)\leq i\leq m\ a(i)\underset{\sim}{\subset} b\}.
  \end{equation*}

  Then $p\leq q$ iff there exists $\vec{a}$ such that $p\leq^* q\oplus\vec{a}$. 
 \end{defin}

\begin{teo}[Gitik, Sharon]
\label{diagonal}
 Let $G$ generic for $\mathbb{P}$ and $H$ generic for $\mathbb{Q}$ as above. Then $V[H][G]\vDash 2^\kappa>\kappa^+\ \wedge\ \neg AP_\kappa\ \wedge\ VGS_\kappa$.
\end{teo}

\begin{prop}
 $\mathbb{Q}$ as above is $\kappa$-geometric.
\end{prop}
\begin{proof}
 The *-Prikry condition is well known, so the proof is as Lemma \ref{Prikry}.
\end{proof}

\begin{cor}
\label{AP}
 Suppose $I0(\kappa,\lambda)$. Then there exists a generic extension in which there exists $j:V_{\kappa+1}\prec V_{\kappa+1}$, $2^{\kappa}>\kappa^+$, there is a very good scale in $\kappa$ but the approachability property does not hold in $\kappa$.
\end{cor}
\begin{proof}
  Let $H$ be $\mathbb{E}$-generic for $V$. By \ref{general} (without the Prikry part) $I0(\kappa,\lambda)$ still holds in $V[H]$, say witnessed by $k$. Then 
  \begin{equation*}
   U_\omega=\{X\supseteq {\cal P}_\kappa(\kappa^{+\omega+2}):k``\kappa^{+\omega+2}\in k(X)\}, 
  \end{equation*}
  defined in $V[H]$, witnesses that $\kappa=\crt(k)$ is $\kappa^{+\omega+2}$- supercompact. Therefore we can force on $\kappa$ with $\mathbb{Q}$. $\mathbb{Q}\subseteq\Pi_{n\in\omega}{\cal P}_\kappa(\kappa^{+n})$, so $\mathbb{Q}\subseteq V_{\kappa_1}$. The hypotheses of Theorem \ref{general} are then satisfied, and Theorem \ref{diagonal} proves the Corollary.
\end{proof}

In \cite{Neeman} Neeman introduced a variation on Gitik-Sharon forcing, that has a more structured preparation forcing and needs more large cardinal power. The result will involve the Tree Property:

 \begin{defin}
  Let $\kappa$ be a cardinal. Then the tree property holds at $\kappa$, $TP(\kappa)$, if every tree of height $\kappa$ and such that all levels have size $<\kappa$ has a cofinal branch.
 \end{defin}

Suppose $j$ witnesses $I0(\kappa,\lambda)$ and let $\langle\kappa_i:i\in\omega\rangle$ be the critical sequence of $j$.

\begin{lem}
\label{limit}
  $V_\lambda\vDash\kappa_1$ is limit of supercompact cardinals.
\end{lem}
\begin{proof}
 It is by reflection of rank-into-rank embeddings: for any $\gamma<\kappa$, the sentence $''\exists k:V_\lambda\prec V_\lambda,j(\gamma)<\crt(k)<j(\kappa)``$ is true, witnessed by $j\upharpoonright V_\lambda$ (note that $j(\gamma)=\gamma$). Then, by elementarity, there exists $k:V_\lambda\prec V_\lambda$ with critical point between $\gamma$ and $\kappa$. Such critical point is supercompact in $V_\lambda$, and choosing different $\gamma$'s we have that the cardinals supercompact in $V_\lambda$ form an unbounded subset of $\kappa$. By elementarity, this is true also for $\kappa_1$.
\end{proof}

Let $\mu_0=\kappa$ and $\mu_{i+1}$ the smallest cardinal supercompact in $V_\lambda$ larger than $\mu_i$, and let $\nu=\sup_{i\in\omega}\mu_i$. By the lemma above, $\nu<\kappa_1$.

Suppose GCH. 
\begin{prop}[Shi \cite{Shi}]
 There is a generic extension of $V$ such that $I0(\kappa,\lambda)$ holds and if $\mu$ is a cardinal supercompact in $V_\lambda$, $\mu$ is indestructible by $\mu$-directed closed forcing. 
\end{prop}

So, we can suppose that all the $\mu_i$ are closed under $\mu_i$-directed closed forcing.

Let $\mathbb{A}_\lambda$ the reverse Easton iteration that is not trivial only on $V_\lambda$-supercompact cardinals limits of $V_\lambda$-supercompact cardinals, and if $\eta$ is such a cardinal, $\mathbb{Q}_\eta$ is the forcing that adds $\nu(\eta)^{++}$ subsets to $\eta$, with conditions of size $<\eta$, where $\nu(\eta)$ is the sup of the $\omega$ $V_\lambda$-supercompact cardinals above $\eta$. As noted before, if $\eta<\lambda$ then $\nu(\eta)<\lambda$.
\begin{itemize}
 \item $\mathbb{A}_\lambda$ is directed closed, because each $\mathbb{Q}_\eta$ is $<\eta$-directed closed;
 \item $\mathbb{A}_\lambda$ is $\lambda$-bounded, because for each $\eta$, $|\mathbb{Q}_\eta|=\nu(\eta)^\eta<\lambda$, and $\lambda$ is strong limit;
 \item $\mathbb{A}_\lambda$ is $j$-coherent, because its definition depends only on $\lambda$.
\end{itemize}

Let $E$ be generic for $\mathbb{A}_\lambda$. Then in $V[E]$:
\begin{itemize}
 \item by Theorem \ref{general}, $I0(\kappa,\lambda)$ holds;
 \item by indestructibility, $\kappa$ is $V_\lambda$-supercompact, and since the forcing is trivial from $\kappa+1$ to $\nu$, and closed enough, $2^\kappa=\nu^{++}$. 
 \item by Gitik-Sharon \cite{GitikSharon}, there exists a $\nu^+$ supercompactness measure on $\kappa$.
\end{itemize}

We say that $\pi$ is a $\nu^+$ supercompactness measure on $\kappa$ if $\pi:V[E]\prec M$, $\crt(\pi)=\kappa$ and $M\upharpoonright\pi(\kappa)=\{\pi(f)(\kappa):f:{\cal P}_{\kappa}(\nu^+)\to\kappa\}$.

Let $U$ be the $\nu^+$ supercompactness measure on $\kappa$, and $U_n$ the $\mu_n$ supercompactness measure on $\kappa$ induced by $U$, i.e., $X\in U_n$ iff $\pi"'\mu_n\in\pi(X)$.

Now the definition of the forcing is the same as \ref{diagonalprikry}, with $\mu_n$ instead of $\kappa^{+n}$:

Let $a,b\in{\cal P}_\kappa(\mu_n)$ and $b\cap\kappa\in\kappa$. Set
 \begin{equation*}
  a\underset{\sim}{\subset}b \leftrightarrow a\subseteq b \wedge \ot(a)<b\cap\kappa.
 \end{equation*}
 
 \begin{defin}
  $p\in\mathbb{Q}$ iff $p=\langle a_0^p,a_1^p,\dots,a_{n-1}^p,X_n^p,X_{n+1}^p,\dots\rangle$ where
  \begin{enumerate}
   \item $\forall l<n\ a_l^p\in{\cal P}_\kappa(\mu_n)$ and $a_l^p\cap\kappa$ is an inaccessible cardinal;
   \item $\forall m\geq n$, $X_m^p\in U_m$;
   \item $\forall m\geq n\ \forall b\in X_m^p\ \forall l<n\ a_l^p\underset{\sim}{\subset} b$;
   \item $\forall i<j<l$ $a_i^p\underset{\sim}{\subset} a_j^p$.
  \end{enumerate}
 \end{defin}

 For $p=\langle a_0^p,a_1^p,\dots,a_{n-1}^p,X_n^p,X_{n+1}^p,\dots\rangle$, let us denote $n$ as $l(p)$. Moreover, for any collection of $A_i\subseteq{\cal P}_\kappa(\kappa^{+i})$, let
 \begin{equation*}
  \tilde{\Pi}_{n\in m}A_n=\{\langle a_0,\dots, a_{m-1}\rangle:\forall i<j<m\ a_i\in A_i\wedge a_i \underset{\sim}{\subset}a_j\}.
 \end{equation*}
 
 For any collection of $A_a$, $a\in{\cal P}_\kappa(\mu_n)$, let 
 \begin{equation*}
  \Delta A_a=\{b\in{\cal P}_\kappa(\mu_n):\forall a\in{\cal P}_\kappa(\mu_n)\ a\underset{\sim}{\subset} b\rightarrow b\in X_a\}. 
 \end{equation*}
 It is a standard result that if each $A_a\in U_n$, then $\Delta A_a\in U_n$.
 
 \begin{defin}
  Let $p,q\in\mathbb{Q}$. Then $p\leq^*q$ iff
  \begin{enumerate}
   \item $l(p)=l(q)$;
   \item $\forall l<l(p)\ a_l^p=a_l^q$;
   \item $\forall m\geq l(p)\ X_m^p\subseteq X_m^q$.
  \end{enumerate}
 \end{defin}

Note that if $G$ is generic for $\mathbb{Q}$ as above as defined in $V[E]$, $(2^{\kappa})^{V[E][G]}=(\nu^{++})$. As $\nu$ is collapsed to $\kappa$, and no other cardinal is collapsed, $V[E][G]\vDash 2^\kappa=\kappa^{++}$

\begin{teo}[Neeman]
 If $G$ is generic for $\mathbb{Q}$ as above as defined in $V[E]$, then $TP(\kappa^+)$.
\end{teo}

\begin{lem}
 $\mathbb{Q}$ as above is $\kappa$-geometric.
\end{lem}
\begin{proof}
 The proof is exactly the same as in \ref{diagonal}
\end{proof}

 \begin{cor}
 \label{TP+}
  Suppose $I0(\kappa,\lambda)$. Then there exists a generic extension in which $I1(\kappa)+2^\kappa>\kappa^+ +TP(\kappa^+)$ holds.
 \end{cor}
 \begin{proof}
  The remarks above show that the hypotheses for Theorem \ref{general} are satisfied: there are three preparation forcing (one for GCH, one for the indesctructibility of supercompactness, and one for blowing up the power of $\kappa$) and they are all reverse Easton iterations with the properties needed. The forcing $\mathbb{Q}$ is a subset of $\Pi_{n\in\omega}{\cal P}_\kappa(\mu_n)$, therefore in $V[E]_{\kappa_1}$ by Lemma \ref{limit}, and it is $\kappa$-geometric. So in $V[E][H]$ $I1(\kappa)$ holds, but also (see above) $2^\kappa=\kappa^{++}$ and $TP(\kappa^+)$
 \end{proof}

\section{Tree property at the double successor}
 Some results can be achieved using the general procedure without much further effort. This is the case for the forcing introduced by Dobrinen and Friedman in \cite{DobFried}, to prove the tree property at a double successor of a singular cardinal.

  Note that for $TP(\kappa^{++})$ to hold, it must be that $2^\kappa>\kappa^{++}$, so it is natural to ask whether $I1(\kappa)$ holds at the same time. This is another property that is implied by $\Box_\kappa$.
 
  \begin{defin}
  For any $\kappa$ inaccessible, the forcing $\Sacks(\kappa)$ is the set of subsets of $2^{<\kappa}$ such that:
  \begin{itemize}
   \item $s\in p$, $t\subseteq s\rightarrow t\in p$;
   \item each $s\in p$ has a proper extension in $p$;
   \item for any $\alpha<\kappa$, if $\langle s_\beta:\beta<\alpha\rangle$ is a $\subseteq$-increasing sequence of elements of $p$, then $\bigcup_{\beta<\alpha}s_\beta\in p$;
   \item there exists a club $C(p)$ such that $\{s\in p: s^\frown 0\in p\ \wedge\ s^\frown 1\in p\}=\{s\in p:\lh(s)\in C(p)\}$.
  \end{itemize}
  Extension is simply the inclusion.
 \end{defin}

 \begin{defin}
  For any $\kappa$ inaccessible and $\gamma(\kappa)$ the first weakly compact above it, $\Sacks^+(\kappa)$ is the $\gamma(\kappa)$ iteration of $\Sacks(\kappa)$ with supports of size $\leq\kappa$.
 \end{defin}
 
 Let $\mathbb{P}$ be the reverse Easton forcing of length $\lambda$ such that $\mathbb{P_\alpha}=\Sacks^+(\alpha)$. This forcing is clearly $\lambda$-bounded and $j$-coherent. Fact 2.7 in \cite{DobFried} states that it is also closed. Moreover, Theorem 3.2 in the same paper shows that in the extension the Tree Property holds in the extension for $\alpha^{++}$, for any $\alpha$ inaccessible. 
 
 \begin{cor}
 \label{TP}
  Suppose $I0(\kappa,\lambda)$. There exists a generic extension of $V$ such that $I1(\kappa)$+$TP(\kappa^{++})$ holds.
 \end{cor}
 \begin{proof}
  Let $\mathbb{P}$ be $\Sacks(\kappa)$ and $\mathbb{Q}$ the Prikry forcing on $\kappa$. By Theorem \ref{general} and the remarks above, it suffices to show that $\mathbb{Q}$ is $\kappa$-geometric and it does not kill the tree property in $\kappa^{++}$. The first is Lemma \ref{Prikry}, the second is Theorem 2 in \cite{HalFried}, and we're done.
 \end{proof}
 
 \section{Open questions}
  The general procedure introduced in Theorem \ref{general} has its own shortcomings. Among the many ''Prikry-like`` forcings, there are some that exploit the full supercompactness of one or many cardinal. A priori, this is not immediately useful: under $I0(\kappa,\lambda)$, $\kappa$ is just $\lambda$-supercompact. Also, in \cite{DimFried} there is a proof that it is consistent to have $\kappa$ less than the least supercompact, so it cannot be a consequence of I0. Thus, we can ask this:
  
  \begin{Q}
   Is I0 + $L(V_{\lambda+1})\vDash(\kappa$ is $\lambda^+$-supercompact) consistent? If so, what is its consistency strength?
  \end{Q}
  Note that we want $\kappa$ to be supercompact in $L(V_{\lambda+1})$ because the Prikry forcing in the general procedure must be in $L(V_{\lambda+1})$, and as $L(V_{\lambda+1})$ does not satisy \AC, but just $\DC_\lambda$, asking for more than $\lambda^+$-supercompactness can be improper.
  
  It is also possible that there are ways to make the general procedure, always or just in some cases, obsolete. For now, there is no proof that I0 is needed for the consistency of I1 and the combinatorial properties above. The usual large cardinals analysis, in fact, many times has results that have the same large cardinal consistency: this case is different because, while usually one starts with a model with a large cardinal, forces the combinatorial property and proves that the large cardinal is intact, in this case the forcing ''reflects`` the large cardinal to a cardinal that had already the property desired. So we can ask:
  
  \begin{Q}
   Is it possible to have the results in Corollaries \ref{nonsch}, \ref{AP}, \ref{TP+} and \ref{TP} with hypotheses weaker than I0? Or is it possible to have the consequences with hypotheses stronger than I1?
  \end{Q}
  
  With generic absoluteness, it is already possible to raise I1 to $j:L_\alpha(V_{\lambda+1})\prec L_\alpha(V_{\lambda+1})$, with $\alpha<\lambda$. An improvement of generic absolutness could improve also this, up to the so-called ''internal I0``, i.e., the existence for any $\alpha<\Theta$ of $j:L_\alpha(V_{\lambda+1})\prec L_\alpha(V_{\lambda+1})$, but for I0 a different approach could be needed.

\end{document}